\documentclass[reqno]{amsart}
\usepackage{amsmath,amsthm,amstext,amssymb}
\usepackage{graphicx}

\newcommand{\Ord}{\mathrm{Ord}}

\usepackage[textsize=footnotesize,color=green!40, bordercolor=white]{todonotes}

\usepackage{hyperref} 

\usepackage[arrow, matrix, curve]{xy}

\makeatletter
\@namedef{subjclassname@2020}{%
  \textup{2020} Mathematics Subject Classification}
\makeatother

\newenvironment{enumerate-(a)}{\begin{enumerate}[label={\upshape (\alph*)}, leftmargin=2pc]}{\end{enumerate}}
\newenvironment{enumerate-(a)-r}{\begin{enumerate}[label={\upshape (\alph*)}, leftmargin=2pc,resume]}{\end{enumerate}}
\newenvironment{enumerate-(A)}{\begin{enumerate}[label={\upshape (\Alph*)}, leftmargin=2pc]}{\end{enumerate}}
\newenvironment{enumerate-(A)-r}{\begin{enumerate}[label={\upshape (\Alph*)}, leftmargin=2pc,resume]}{\end{enumerate}}
\newenvironment{enumerate-(i)}{\begin{enumerate}[label={\upshape (\roman*)}, leftmargin=2pc]}{\end{enumerate}}
\newenvironment{enumerate-(i)-r}{\begin{enumerate}[label={\upshape (\roman*)}, leftmargin=2pc,resume]}{\end{enumerate}}
\newenvironment{enumerate-(I)}{\begin{enumerate}[label={\upshape (\Roman*)}, leftmargin=2pc]}{\end{enumerate}}
\newenvironment{enumerate-(I)-r}{\begin{enumerate}[label={\upshape (\Roman*)}, leftmargin=2pc,resume]}{\end{enumerate}}
\newenvironment{enumerate-(1)}{\begin{enumerate}[label={\upshape (\arabic*)}, leftmargin=2pc]}{\end{enumerate}}
\newenvironment{enumerate-(1)-r}{\begin{enumerate}[label={\upshape (\arabic*)}, leftmargin=2pc,resume]}{\end{enumerate}}

\theoremstyle{plain}
\newtheorem{theorem}{Theorem}[section]
\newtheorem{lemma}[theorem]{Lemma}
\newtheorem{question}[theorem]{Question}

\newtheorem{corollary}[theorem]{Corollary}
\newtheorem{proposition}[theorem]{Proposition}

\newtheorem*{claim*}{Claim}

\theoremstyle{definition}
\newtheorem{definition}[theorem]{Definition}

\newcommand{\betrag}[1]{\vert{#1}\vert}
\newcommand{\height}[1]{{\rm{ht}}(#1)}

\newcommand{\dom}[1]{{{\rm{dom}}(#1)}}

\newcommand{\cof}[1]{{{\rm{cof}}(#1)}}
\newcommand{\otp}[1]{{{\rm{otp}}\left(#1\right)}}

\newcommand{\length}[2]{{\rm{lh}}_{{#2}}({#1})}

\newcommand{\map}[3]{{#1}:{#2}\longrightarrow{#3}}

\newcommand{\Set}[2]{\{{#1}~\vert~{#2}\}}
\newcommand{\seq}[2]{\langle{#1}~\vert~{#2}\rangle}

\newcommand{\goedel}[2]{{\prec}{#1},{#2}{\succ}}

\newcommand{\Add}[2]{{\rm{Add}}({#1},{#2})}
\newcommand{\Coll}[2]{{\rm{Col}}({#1},{#2})}

\newcommand{\id}{{\rm{id}}}
\newcommand{\Lim}{{\rm{Lim}}}

\newcommand{\LL}{{\rm{L}}}
\newcommand{\ZFC}{{\rm{ZFC}}}
\newcommand{\GCH}{{\rm{GCH}}}

\newcommand{\PPP}{{\mathbb{P}}}

\newcommand{\RRR}{{\mathbb{R}}}

\newcommand{\TTT}{{\mathbb{T}}}

\newcommand{\VV}{{\rm{V}}}

\newcommand{\CH}{{\rm{CH}}}

\title{Descriptive properties of higher Kurepa trees}

\author{Philipp L\"ucke}
\address{Institut de Matem\`{a}tica, Universitat de Barcelona. 
Gran via de les Corts Catalanes 585,
08007 Barcelona, Spain.}
\email{philipp.luecke@ub.edu}

\author{Philipp Schlicht}
\address{Institute for Mathematics, University of Vienna, Kolingasse 14-16, 1090 Vienna, Austria 
and 
School of Mathematics, University of Bristol, Fry Building, Woodland Road, Bristol, BS8 1UG, UK} 
\email{philipp.schlicht@univie.ac.at}

\thanks{This project has received funding from the European Union’s Horizon 2020 research and innovation programme under the Marie Sk{\l}odowska-Curie grant agreements No 842082 of the first author (Project \emph{SAIFIA: Strong Axioms of Infinity -- Frameworks, Interactions and Applications}) and No 794020 of the second author (Project \emph{IMIC: Inner models and infinite computations}). The second author was partially supported by FWF grant number I4039. } 

\subjclass[2020]{03E05, 03E35, 03E47.}
\keywords{Generalized Baire spaces, Kurepa trees, Continuous images}

\begin{document}

\begin{abstract} 
 We use generalizations of concepts from descriptive set theory 
  to study combinatorial objects of uncountable regular cardinality, focussing on higher Kurepa trees and the representation of the sets of cofinal branches through such trees as continuous images of function spaces. 
For different types of uncountable regular cardinals $\kappa$, our results provide a complete picture of all consistent scenarios  for  the representation of sets of cofinal branches through $\kappa$-Kurepa trees as retracts of the generalized Baire space ${}^\kappa\kappa$ of $\kappa$. 
 In addition, these results can be used to determine the consistency of most of the corresponding statements for continuous images of ${}^\kappa\kappa$.

 %
%
\end{abstract}

\maketitle



\section{Introduction}

 Given an infinite regular cardinal $\kappa$, the \emph{generalized Baire space of $\kappa$} consists of the set ${}^\kappa\kappa$ of all functions from $\kappa$ to $\kappa$ equipped with the topology whose basic open sets are of the form $N_s=\Set{x\in{}^\kappa\kappa}{s\subseteq x}$, where $s$ is an element of the set ${}^{{<}\kappa}\kappa$ of all functions $\map{t}{\alpha}{\kappa}$ with $\alpha<\kappa$. 
 One of the most basic structural features of the classical Baire space ${}^\omega\omega$ is the fact that non-empty closed subsets of ${}^\omega\omega$ are retracts of ${}^\omega\omega$ (see {\cite[Proposition 2.8]{MR1321597}}), i.e. for every such that $C$ there is a continuous surjection $\map{r}{{}^\omega\omega}{A}$ with $r\restriction A=\id_A$. 
 In contrast, the results of \cite{luecke-schlicht-continuous-images} show that such results cannot be generalized to higher cardinalities and this failure highlights fundamental differences between $\omega$ and higher regular cardinals, e.g. the existence of limit ordinals below the given cardinal and possible existence of Aronszajn and Kurepa trees at these cardinals. 
 First, {\cite[Proposition 1.4]{luecke-schlicht-continuous-images}} shows that for every uncountable regular cardinal $\kappa$, there is a non-empty closed subset of ${}^\kappa\kappa$ that is not a retract of ${}^\kappa\kappa$. 
 Moreover, {\cite[Theorem 1.5]{luecke-schlicht-continuous-images}} shows that for every uncountable cardinal $\kappa$ satisfying $\kappa=\kappa^{{<}\kappa}$, there is a non-empty closed subset of ${}^\kappa\kappa$ that is not a continuous image of ${}^\kappa\kappa$.\footnote{It is easy to see that, if there exists a cardinal $\mu<\kappa$ with $2^\mu=2^\kappa$, then every non-empty subset of ${}^\kappa\kappa$ is a continuous image of ${}^\kappa\kappa$. In particular, some cardinal arithmetic assumption on $\kappa$ is necessary for the conclusion of {\cite[Theorem 1.5]{luecke-schlicht-continuous-images}}.} 
 In particular, the classes of continuous images of ${}^\kappa\kappa$ and continuous images of non-empty closed subsets of ${}^\kappa\kappa$ in ${}^\kappa\kappa$ do not coincide in this case. Finally, the various results of \cite{luecke-schlicht-continuous-images} strongly motivate an investigation of the class of all closed subsets of ${}^\kappa\kappa$ that are continuous images of ${}^\kappa\kappa$. 
 Since closed subsets of ${}^\kappa\kappa$ canonically  correspond to sets of cofinal branches through trees of height 
  $\kappa$, the study of this class of subsets turns out to be closely connected to the existence of certain combinatorial objects of the given  cardinality and the validity of combinatorial principles at the corresponding cardinals.

 In this paper, motivated by results from \cite{luecke-schlicht-continuous-images} (see Theorem \ref{theorem:SurjImageWideLevels} below),  we focus on the representation of the sets of cofinal branches through $\kappa$-Kurepa trees as continuous images of ${}^\kappa\kappa$. 
 Our results will show that if $\kappa$ is a successor of a cardinal of countable cofinality, then $\kappa$-Kurepa trees provide more examples of closed subsets of ${}^\kappa\kappa$ that are not continuous images of ${}^\kappa\kappa$.  
  In contrast, we will also show that for other types of uncountable regular cardinals, various natural questions about such continuous representations of sets of branches through Kurepa trees can have widely different answers that depend 
  heavily on the underlying model of set theory. 

 Remember that a partial order $\TTT$ is a \emph{tree} if it has a unique minimal element and for every element $t$ of $\TTT$, the set $\mathrm{pred}_\TTT(t)=\Set{s\in\TTT}{s<_\TTT t}$ is well-ordered by $<_\TTT$. Given a tree $\TTT$, $t\in\TTT$ and $\alpha\in\Ord$, we define $\length{t}{\TTT}=\otp{\mathrm{pred}_\TTT(t),<_\TTT}$, $\TTT_\alpha=\Set{s\in\TTT}{\length{s}{\TTT}=\alpha}$, $\TTT_{{<}\alpha}=\bigcup\Set{\TTT_{\bar{\alpha}}}{\bar{\alpha}<\alpha}$ and $\height{\TTT}=\min\Set{\beta\in\Ord}{\TTT_\beta=\emptyset}$. Next, we say that a subset $b$ of a tree $\TTT$ is a \emph{branch} if it $<_\TTT$-downwards-closed and linearly ordered by $<_\TTT$ and it is a \emph{cofinal branch} if in addition $\otp{b,<_\TTT}=\height{\TTT}$ holds. Finally, given an uncountable regular cardinal $\kappa$, a tree $\TTT$ with $\height{\TTT}=\kappa$ is a \emph{$\kappa$-Kurepa tree} if $\betrag{\TTT_\alpha}<\kappa$ holds for all $\alpha<\kappa$ and $\TTT$ has at least $\kappa^+$-many cofinal branches. We usually say \emph{Kurepa tree} instead of $\omega_1$-Kurepa tree. Seminal results of Jensen (see {\cite[Chapter III, Section 3]{MR750828}}) show that, in the constructible universe $\LL$, $\kappa$-Kurepa trees exist for every uncountable regular cardinal $\kappa$.\footnote{Note that for all inaccessible cardinals $\kappa$, the existence of $\kappa$-Kurepa trees is trivial. Therefore, one usually considers trees with stronger restrictions on the size of levels at inaccessible cardinals (see \cite{JensenKunen1969:Ineffable} and Section \ref{subsection:NoImage} below).}
  In contrast, a classical argument of Silver (see {\cite[Section 3]{MR0284331}}) shows that if $\kappa$ is an uncountable regular cardinal that is not inaccessible, $\theta>\kappa$ is inaccessible and $G$ is $\Coll{\kappa}{{<}\theta}$-generic over $\VV$, then there are no $\kappa$-Kurepa trees in $\VV[G]$.

 Fix an infinite regular cardinal $\kappa$ and $0<n<\omega$. A subset $T$ of $({}^{{<}\kappa}\kappa)^n$ is a \emph{subtree of $({}^{{<}\kappa}\kappa)^n$} if $\dom{t_0}=\ldots=\dom{t_{n-1}}$ and $\langle t_0\restriction\alpha,\ldots,t_{n-1}\restriction\alpha\rangle\in T$ for all $\langle t_0,\ldots,t_{n-1}\rangle\in T$ and $\alpha<\dom{t_0}$.  
Note that for every such subtree $T$, the partial order $\TTT_T=\langle T,\lhd\rangle$ with $$\langle s_0,\ldots,s_{n-1}\rangle\lhd\langle t_0,\ldots,t_{n-1}\rangle ~ \Longleftrightarrow ~ s_0\subsetneq t_0 \wedge \ldots \wedge s_{n-1}\subsetneq t_{n-1}$$ for all $\langle s_0,\ldots,s_{n-1}\rangle,\langle t_0,\ldots,t_{n-1}\rangle\in T$ is a tree with $(\TTT_T)_\alpha=T\cap({}^\alpha\kappa)^n$ for all $\alpha<\kappa$. 
 Set $\height{T}=\height{\TTT_T}\leq\kappa$. If $\height{T}=\kappa$, then there is a canonical bijection between the set of cofinal branches through $\TTT_T$ and the set $$[T] ~ = ~ \Set{\langle x_0,\ldots,x_{n-1}\rangle\in({}^\kappa\kappa)^n}{\forall\alpha<\kappa ~ \langle x_0\restriction\alpha,\ldots,x_{n-1}\restriction\alpha\rangle\in T}$$ given by restrictions. It is easy to see that for every subtree $T$ of $({}^{{<}\kappa}\kappa)^n$ of height $\kappa$, the set $[T]$ is closed in the product space $({}^\kappa\kappa)^n$. Moreover, for every non-empty closed subset $C$ of $({}^\kappa\kappa)^n$, the set $$T_C ~ = ~ \Set{\langle t_0,\ldots,t_{n-1}\rangle\in({}^{{<}\kappa}\kappa)^n}{C\cap(N_{t_0}\times\ldots\times N_{t_{n-1}})\neq\emptyset}$$ is a subtree of $({}^{{<}\kappa}\kappa)^n$ with $C=[T_C]$.

Now, let $\kappa$ be an uncountable regular cardinal and let $\TTT$ be a tree of height $\kappa$ with the property that every node in $\TTT$ has at most $\kappa$ many direct successors. If $\TTT$ is \emph{extensional at limit levels} (i.e. if $s=t$ holds for all $s,t\in\TTT$ with $\length{s}{\TTT}=\length{t}{\TTT}\in\Lim$ and $\mathrm{pred}_\TTT(s)=\mathrm{pred}_\TTT(t)$), then it is easy to construct a subtree $T$ of ${}^{{<}\kappa}\kappa$ with the property that the tree $\TTT$ and $\TTT_T$ are isomorphic. In general, we can consider the tree $\bar{\TTT}$ that consists of all non-empty branches $b$ through $\TTT$ with the property that there exists $t\in\TTT$ with $s\leq_\TTT t$ for all $s\in b$ and is ordered by inclusion. Then $\bar{\TTT}$ is extensional at limit levels, the map $[t\mapsto \mathrm{pred}_\TTT(t)\cup\{t\}]$ is an isomorphism between $\TTT$ and a cofinal subtree of $\bar\TTT$ and hence there is a canonical isomorphism between the sets of cofinal branches through $\TTT$ and $\bar{\TTT}$. Moreover, it is easy to see that the assumption that $\TTT$ is a $\kappa$-Kurepa tree implies that $\bar{\TTT}$ is a $\kappa$-Kurepa tree too. In combination, this shows that the existence of a $\kappa$-Kurepa tree is equivalent to the existence of a \emph{$\kappa$-Kurepa subtree of ${}^{{<}\kappa}\kappa$}, i.e. a subtree $T$ of ${}^{{<}\kappa}\kappa$ with the property that the tree $\TTT_T$ is a $\kappa$-Kurepa tree.


\subsection{Kurepa trees that are not continuous images}\label{subsection:NoImage}

The following result provides several scenarios in which the closed sets induced by Kurepa trees are not continuous images of the corresponding generalized Baire space. 
 The statement of Corollary \ref{corollary: omega1 Kurepa trees} below was our original motivation for the work presented in this paper.

\begin{theorem}\label{theorem:Main-Negative}
 If $\kappa$ is an uncountable regular cardinal with $\mu^\omega\geq\kappa$ for some $\mu<\kappa$ and $T$ is a $\kappa$-Kurepa subtree of ${}^{{<}\kappa}\kappa$ with $\betrag{[T]}>\kappa^{{<}\kappa}$, then the set $[T]$ is not a continuous image of ${}^\kappa\kappa$. 
\end{theorem}


This theorem has the following two direct corollaries:

\begin{corollary}\label{corollary: omega1 Kurepa trees} 
 Assume that $\CH$ holds. If $T$  is a subtree of ${}^{{<}\omega_1}\omega_1$ that is a Kurepa tree, then $[T]$ is not a continuous image of ${}^{\omega_1}\omega_1$. \qed
\end{corollary}



\begin{corollary}
\label{corollary: mu+-Kurepa trees for singular mu} 
 Assume that $\kappa=\mu^+=2^\mu$ for some singular cardinal $\mu$ of countable cofinality. If $T$ is a subtree of ${}^{{<}\kappa}\kappa$ that is a $\kappa$-Kurepa tree, then $[T]$ is not a continuous image of ${}^{\kappa}\kappa$. \qed 
\end{corollary}


Note that the above notion of Kurepa trees trivializes at inaccessible cardinals $\kappa$, because the complete binary tree ${}^{{<}\kappa}2$ obviously satisfies the listed properties.  
  Moreover, it is easy to see that the unique  function $\map{r}{{}^\kappa\kappa}{{}^\kappa 2}$ with $r(x)(\alpha)=\min\{1,x(\alpha)\}$ for all $x\in{}^\kappa\kappa$ and $\alpha<\kappa$ is a retraction from ${}^\kappa\kappa$ to ${}^\kappa 2$ in this case. 
 In the other direction, it is also possible to use a result from \cite{luecke-schlicht-continuous-images} to show that not every Kurepa tree at an inaccessible cardinal is a continuous image. 
 Moreover, this statement can be extended to certain regular cardinals that are only inaccessible in some inner model computing the successor of the given cardinal correctly.

 \begin{theorem}\label{theorem:KurepaInaccessibleWO}
 Let $M$ be an inner model and let $\kappa$ be a cardinal with $\kappa=\kappa^{{<}\kappa}$. 
  If $\kappa$ is inaccessible in $M$ and $(\kappa^+)^M=\kappa^+$ holds, then there is a $\kappa$-Kurepa subtree $T$ of ${}^{{<}\kappa}\kappa$ with $T\subseteq({}^{{<}\kappa}2)^M$ and the property that the set $[T]$ is not a continuous image of ${}^\kappa\kappa$. 
 \end{theorem}

 Motivated by these results about inaccessible cardinals, we also consider the following strengthening of the definition of Kurepa trees: given an uncountable regular cardinal $\kappa$, a tree $T$ of height $\kappa$ is called \emph{slim} if $\betrag{T(\alpha)}\leq\betrag{\alpha}$ holds for co-boundedly many $\alpha<\kappa$. 
 Classical results of Jensen and Kunen in \cite{JensenKunen1969:Ineffable} show that there are no slim $\kappa$-Kurepa trees at ineffable cardinals $\kappa$ and, in the constructible universe $\LL$, slim $\kappa$-Kurepa trees exist at every uncountable regular cardinal $\kappa$ that is not ineffable. 

The proof of Theorem \ref{theorem:Main-Negative} also allows us to derive the following statement.

\begin{theorem}\label{theorem:Kurepa tree at inaccessibles} 
 If $\kappa$ is an inaccessible cardinal and $T$ is a slim Kurepa subtree of ${}^{{<}\kappa}\kappa$, then $[T]$ is not a continuous image of ${}^{\kappa}\kappa$. 
\end{theorem}

Finally, the results of this paper allow us to use results of Donder from \cite{MR791060} and Velleman from \cite{MR771773} to show that the absence of large cardinals in the constructible universe implies the existence of Kurepa trees whose induced closed subsets are not continuous images. 

\begin{theorem}\label{theorem:ConsStrengthContImage}
 Let $\kappa$ be an uncountable regular cardinal such that  $\kappa=\kappa^{{<}\kappa}$ holds and neither $\kappa$ nor $\kappa^+$ are inaccessible in $\LL$.   
  Then there is a $\kappa$-Kurepa subtree of ${}^\kappa\kappa$ with the property that $[T]$ is not a continuous image of ${}^\kappa\kappa$. 
\end{theorem}


\subsection{Kurepa trees that are continuous images}\label{subsection:PositiveResults}

Somewhat surprisingly, Corollaries \ref{corollary: omega1 Kurepa trees} and \ref{corollary: mu+-Kurepa trees for singular mu}
in previous section turn out to be the only provable restrictions on the existence of $\kappa$-Kurepa trees whose sets of cofinal branches are continuous images of ${}^\kappa\kappa$. 
The proof of the following theorem relies on a result of Donder on the structural properties of the Kurepa trees constructed from the canonical morasses at successor cardinals in the constructible universe. 
This result can be combined with Theorem \ref{theorem:Main-Negative} to show that the statement that there is an $\omega_2$-Kurepa subtree $T$ of ${}^{{<}\omega_2}\omega_2$ with the property that $[T]$ is a continuous image of ${}^{\omega_2}\omega_2$ is independent of the axioms of $\ZFC$ by considering the constructible universe $\LL$ and models of the negation of the Continuum Hypothesis.

\begin{theorem}\label{theorem:KurepaImageInL}
 Assume that $\VV=\LL$ holds. Then the following statements are equivalent for every uncountable regular cardinal $\kappa$: 
 \begin{enumerate}
  \item The cardinal $\kappa$ is not the successor of a cardinal of countable cofinality. 

  \item There is a $\kappa$-Kurepa subtree of ${}^{{<}\kappa}\kappa$ with the property that $[T]$ is a retract of ${}^\kappa\kappa$. 

  \item There is a $\kappa$-Kurepa subtree of ${}^{{<}\kappa}\kappa$ with the property that $[T]$ is a continuous image of ${}^\kappa\kappa$. 
 \end{enumerate}
\end{theorem}

In combination with Theorem \ref{theorem:Kurepa tree at inaccessibles}, the previous result directly implies an analogous characterization for slim Kurepa trees.

\begin{corollary}
 Assume that $\VV=\LL$ holds. Then the following statements are equivalent for every uncountable regular cardinal $\kappa$: 
 \begin{enumerate}
  \item The cardinal $\kappa$ is neither inaccessible nor  the successor of a cardinal of countable cofinality. 

  \item There is a slim $\kappa$-Kurepa subtree of ${}^{{<}\kappa}\kappa$ with the property that $[T]$ is a retract of ${}^\kappa\kappa$. 

  \item There is a slim $\kappa$-Kurepa subtree of ${}^{{<}\kappa}\kappa$ with the property that $[T]$ is a continuous image of ${}^\kappa\kappa$. \qed
 \end{enumerate}
\end{corollary}

The next result shows that the positive implications of Theorem \ref{theorem:KurepaImageInL} for successors of regular cardinals can also be obtained by collapsing an inaccessible cardinal to become the successor of an uncountable regular cardinal.

\begin{theorem}\label{theorem:CollapseImageNONImage}
 If $\kappa$ is an inaccessible cardinal and $\mu<\kappa$ is an uncountable regular cardinal, then 
there is a generic extension $\VV[G]$ of the ground model $\VV$ that preserves all cofinalities less than or equal to $\mu$ and greater than or equal to $\kappa$, such that the following statements hold in $\VV[G]$: 
 \begin{enumerate}
  \item $\kappa=\mu^+$. 
  
  \item There is a $\kappa$-Kurepa subtree $T_0$ of ${}^{{<}\kappa}\kappa$ with the property that the set $[T_0]$ is a retract of ${}^\kappa\kappa$. 
  
  \item There is a $\kappa$-Kurepa subtree $T_1$ of ${}^{{<}\kappa}\kappa$ with $T_1\subseteq T_0$ and the property that the set $[T_1]$ is not a continuous image of ${}^\kappa\kappa$. 
 \end{enumerate}
\end{theorem}

The notion of \emph{slimness} considered in Theorem \ref{theorem:Kurepa tree at inaccessibles} has a natural weakening that only requires $\betrag{T(\alpha)}\leq\betrag{\alpha}$ to hold in a stationary set of $\alpha<\kappa$. 
 We refer to this property as \emph{stationary slimness}. 
The following result shows that, in general, it is not possible to replace \emph{slimness} by \emph{stationary slimness} in the assumptions of Theorem \ref{theorem:Kurepa tree at inaccessibles}. 
 Remember that an inaccessible cardinal $\kappa$ is a \emph{$2$-Mahlo cardinal} if the set of Mahlo cardinals less than $\kappa$ is stationary in $\kappa$.

\begin{theorem}\label{theorem:StatSlimImage}
 If $\kappa$ is a $2$-Mahlo cardinal, then the following statements hold in a generic extension of the ground model $\VV$: 
 \begin{enumerate}
  \item $\kappa$ is inaccessible. 

  \item There is a stationary slim $\kappa$-Kurepa subtree $T$ of ${}^{{<}\kappa}\kappa$ with the property  that $[T]$ is a retract of ${}^\kappa\kappa$. 
 \end{enumerate}
\end{theorem}


\subsection{Kurepa trees that are not retracts}\label{subsection:NoRetracts}

The trees constructed in the proofs of the results of the previous section have many isolated points. 
The next result shows that this is a necessary condition. 
Note that, throughout this paper, \emph{inaccessible} means \emph{strongly inaccessible} and, if $\kappa$ is an inaccessible cardinal, then ${}^{{<}\kappa}2$ is a $\kappa$-Kurepa subtree of ${}^{{<}\kappa}\kappa$ and the set $[{}^{{<}\kappa}2]={}^\kappa 2$ is a retract of ${}^\kappa\kappa$ without isolated points.

\begin{theorem}\label{theorem:ResultsKurepaNoIsolated}
 Let $\kappa$ be an uncountable regular cardinal and let $T$ be a $\kappa$-Kurepa subtree of ${}^\kappa\kappa$. Assume that either $\kappa$ is not inaccessible or $T$ is stationary slim.   If the set $[T]$ is a retract of ${}^\kappa\kappa$, then it contains isolated points.  
\end{theorem}

This result allows us to show that the existence of Kurepa trees implies the existence of Kurepa trees that are not retracts.

\begin{theorem}\label{theorem:NotAllRetracts}
 Let $\kappa$ be an uncountable regular cardinal with $\kappa^{<\kappa}=\kappa$. If there is a $\kappa$-Kurepa tree $S$, then there is a $\kappa$-Kurepa subtree $T$ of ${}^{{<}\kappa}\kappa$ with the property that the set $[T]$ is not a retract of ${}^\kappa\kappa$. Moreover, if $S$ is stationary slim, then $T$ can be taken to be stationary slim. 
\end{theorem}

 Finally, our techniques allow us to show that for higher Kurepa trees, the property of being a continuous image neither implies the existence of isolated branches nor the property of being a retract.

\begin{theorem}\label{theorem:ContImageButNotRetracts}
 Let $\kappa$ be an uncountable regular cardinal. If there is a $\kappa$-Kurepa tree $S$ with the property that the set $[S]$ is a continuous image of ${}^\kappa\kappa$, then there is a $\kappa$-Kurepa subtree $T$ of ${}^{{<}\kappa}\kappa$ with the property that the set $[T]$ does not contain isolated points, it is a continuous image of ${}^\kappa\kappa$ and it is not a retract of ${}^\kappa\kappa$. 
 Moreover, if $S$ is stationary slim, then $T$ can be taken to be stationary slim. 
\end{theorem}


\section{Wide subtrees}

The following result from \cite{luecke-schlicht-continuous-images} will be our main tool for showing that the closed subsets induced by certain Kurepa trees are not continuous images of the whole space.

 \begin{theorem}[{\cite[Theorem 7.1]{luecke-schlicht-continuous-images}}]\label{theorem:SurjImageWideLevels}
  Let 
  $\kappa$ be an uncountable regular, let $A$ be an unbounded subset of $\kappa$ and let $T$ be a subtree of ${}^{{<}\kappa}\kappa$. 
 If $\mu$ is a cardinal with $\mu^{{<}\kappa}<\betrag{[T]}$ and $\map{c}{{}^\kappa\mu}{[T]}$ is a continuous surjection, then there is a strictly increasing sequence $\seq{\lambda_n\in A}{n<\omega}$ with least upper bound $\lambda$ and an injection
  \begin{equation*}
    \map{i}{\prod_{n<\omega}\lambda_n}{T(\lambda)}.
  \end{equation*}
 such that 
 \begin{equation*}
    x\restriction n = y\restriction n ~ \Longleftrightarrow ~ i(x)\restriction\lambda_n = i(y)\restriction\lambda_n
 \end{equation*}
 holds for all $x,y\in\prod_{n<\omega}\lambda_n$ and all  $n<\omega$. 
 \end{theorem}

The above lemma allows us to provide short proofs of  two theorems presented in Section \ref{subsection:NoImage}.

\begin{proof}[Proof of Theorem \ref{theorem:Main-Negative}]
 Let $\kappa$ be an uncountable regular cardinal with $\mu^\omega\geq\kappa$ for some $\mu<\kappa$ and let $T$ be a $\kappa$-Kurepa subtree of ${}^{{<}\kappa}\kappa$ with $\betrag{[T]}>\kappa^{{<}\kappa}$. 
 Assume, towards a contradiction, that the set $[T]$ is a continuous image of ${}^\kappa\kappa$.  
Since we have $\kappa^{{<}\kappa}<\betrag{[T]}$, we can apply Theorem \ref{theorem:SurjImageWideLevels} to find a strictly increasing sequence $\seq{\lambda_n}{n<\omega}$ of ordinals in the interval $(\mu,\kappa)$ with least upper bound $\lambda$ and an injection $\map{i}{\prod_{n<\omega}\lambda_n}{T(\lambda)}$. 
 Moreover, since $\lambda_n>\mu$ holds for all $n<\omega$, we can conclude that $$\betrag{T(\lambda)} ~ \geq ~ \betrag{\prod_{n<\omega}\lambda_n} ~ \geq ~ \mu^\omega ~ \geq ~ \kappa,$$ contradicting the fact that $T$ is a $\kappa$-Kurepa subtree of ${}^{{<}\kappa}\kappa$. 
\end{proof}

\begin{proof}[Proof of Theorem \ref{theorem:Kurepa tree at inaccessibles}]
 Let $\kappa$ be an inaccessible cardinal and let $T$ be a slim $\kappa$-Kurepa subtree of ${}^{{<}\kappa}\kappa$. Assume, towards a contradiction, that the set $[T]$ is a continuous image of ${}^\kappa\kappa$. 
Pick $\alpha<\kappa$ with the property that $\betrag{T(\beta)}=\betrag{\beta}$ holds for all $\alpha\leq\beta<\kappa$.  Using Theorem \ref{theorem:SurjImageWideLevels}, we find a strictly increasing sequence $\seq{\lambda_n}{n<\omega}$ of cardinals in the interval $(\alpha,\kappa)$ with least upper bound $\lambda$ and an injection $\map{i}{\prod_{n<\omega}\lambda_n}{T(\lambda)}$. Then  K\"onig's Theorem allows us to conclude that $$\betrag{T(\lambda)} ~ \geq ~ \betrag{\prod_{n<\omega}\lambda_n} ~ 
> ~ \betrag{\sum_{n<\omega}\lambda_n} ~ = ~ \lambda,$$ a contradiction. 
\end{proof}

In the remainder of this section, we prove Theorem \ref{theorem:ConsStrengthContImage}. Our arguments will be based on the concept introduced in the next definition.

\begin{definition}(\cite[Section 1]{MR631563}) 
 A tree $T$ has a \emph{Cantor subtree} if there is a strictly increasing sequence $\seq{\lambda_n}{n<\omega}$ with $\lambda=\sup_{n<\omega}\lambda_n<\height{T}$ and an uncountable subset $B$ of $T(\lambda)$ with the property that the set $\Set{s\in T(\lambda_n)}{\exists t\in B ~ s<_T t}$ is countable for every $n<\omega$. 
\end{definition}

The next statement is a direct consequence of Theorem \ref{theorem:SurjImageWideLevels}.

\begin{corollary}\label{corollary:ContImageCantosSubtree}
 Let $\kappa$ be an uncountable regular cardinal and and let $T$ be a subtree of ${}^{{<}\kappa}\kappa$ with $\betrag{[T]}>\kappa^{{<}\kappa}$. 
 If $[T]$ is a continuous image of ${}^\kappa\kappa$, then $T$ contains a Cantor subtree. \qed 
\end{corollary}

\begin{proof}[Proof of Theorem \ref{theorem:ConsStrengthContImage}]
 Let $\kappa$ be an uncountable regular cardinal with $\kappa=\kappa^{{<}\kappa}$ and the property that neither $\kappa$ nor $\kappa^+$ are inaccessible in $\LL$.  
   
  \begin{claim*}
   There is a simplified $(\kappa,1)$-morass with linear limits (see {\cite[Section 2]{MR771773}}). 
  \end{claim*}
  
  \begin{proof}[Proof of the Claim]
   As discussed at the end of \cite{MR1049849}, our assumption allows us to find $A\subseteq\kappa$ such that $\kappa^+=(\kappa^+)^{\LL[A]}$,  Donder's construction of a simplified $(\kappa,1)$-morass with linear limits in \cite{MR791060} can be carried out in $\LL[A]$ and the resulting morass is also a  simplified $(\kappa,1)$-morass with linear limits in $\VV$.  
  \end{proof}
  
  By the above claim, we can apply {\cite[Theorem 4.3]{MR771773}} to find a $\kappa$-Kurepa subtree $T$ of ${}^{{<}\kappa}\kappa$ without Cantor subtrees. By Corollary \ref{corollary:ContImageCantosSubtree}, we know that the closed set $[T]$ is not a continuous image of ${}^\kappa\kappa$. 
\end{proof}


\section{Trees induced by inner models}

This section is devoted to the proof of Theorem \ref{theorem:KurepaInaccessibleWO}. Our arguments are a variation of the proof of {\cite[Theorem 1.5]{luecke-schlicht-continuous-images}}. 
In addition, we will show that the statement of Corollary \ref{corollary:ContImageCantosSubtree} can, in general, not be reversed.

\begin{proof}[Proof of Theorem \ref{theorem:KurepaInaccessibleWO}]
 In the following, we let $\map{\goedel{\cdot}{\cdot}}{\Ord\times\Ord}{\Ord}$ denote the \emph{G\"odel pairing function}. Given an ordinal $\gamma$ closed under $\goedel{\cdot}{\cdot}$ and $x\in{}^\gamma 2$, we define $<_x$ to be the unique binary relation on $\gamma$ with $$\alpha<_x\beta ~ \Longleftrightarrow ~ x(\goedel{\alpha}{\beta})=1$$ for all $\alpha,\beta<\gamma$.  
 
 Let $\kappa$ be a cardinal satisfying $\kappa=\kappa^{{<}\kappa}$ and let $M$ be an inner model with the property that $\kappa$ is inaccessible in $M$ and $(\kappa^+)^M=\kappa^+$ holds. Work in $M$ and set $$W ~ = ~ \Set{x\in{}^\kappa 2}{\textit{$(\kappa,<_x)$ is a well-order}}.$$ Then it is easy to see that the fact that $\kappa$ is uncountable and regular implies that $W$ is a closed subset of ${}^\kappa\kappa$ and hence there is a subtree $T$ of ${}^{{<}\kappa}2$ with $W=[T]$.  Since $\betrag{W}=2^\kappa$, $T\subseteq{}^{{<}\kappa}2$ and $\kappa$ is inaccessible, we can conclude that $T$ is a $\kappa$-Kurepa subtree of ${}^{{<}\kappa}\kappa$. 

 Now, work in $\VV$. Then our assumptions on $M$ imply that $T$ is still a $\kappa$-Kurepa subtree of ${}^{{<}\kappa}\kappa$. Moreover, if $x\in[T]$, then the regularity of $\kappa$ implies that $(\kappa,<_x)$ is a well-order. 
 Given $x\in[T]$ and $\alpha<\kappa$, we let $\mathrm{rnk}_x(\alpha)$ denote the rank of $\alpha$ in the well-order $(\kappa,<_x)$. 
Assume, towards a contradiction, that the set $[T]$ is a continuous image of ${}^\kappa\kappa$. Then a combination of {\cite[Lemma 2.2]{luecke-schlicht-continuous-images}} with {\cite[Lemma 2.3]{luecke-schlicht-continuous-images}} yields a ${<}\kappa$-closed subtree $U$ of ${}^{{<}\kappa}\kappa\times{}^{{<}\kappa}\kappa$ without end nodes\footnote{A subtree $S$ of $({}^{{<}\kappa}\kappa)^n$ is \emph{${<}\kappa$-closed} if for every $\lambda<\kappa$ and every $\lhd$-increasing sequence $\seq{\langle s_0^\xi,\ldots,s_{n-1}^\xi\rangle}{\xi<\lambda}$ in $S$, the tuple $\langle\bigcup_{\xi<\lambda}s^\xi_0,\ldots,\bigcup_{\xi<\lambda}s^\xi_{n-1}\rangle$ is also an element of $S$. We call an element of such a tree $S$ an \emph{end node} if it is $\lhd$-maximal in $S$.} such that $$[T] ~ = ~ p[U] ~ = ~ \Set{x\in{}^\kappa\kappa}{\exists y ~ \langle x,y\rangle\in[U]}.$$ Note that these properties of $U$ imply that for every $\langle t,u\rangle\in U$, there exists  a pair $\langle x,y\rangle\in[U]$ with $t\subseteq x$ and $u\subseteq y$. Given $\langle t,u\rangle\in U$ and $\alpha<\kappa$, we define $$r(t,u,\alpha) ~ = ~ \sup\Set{\mathrm{rnk}_x(\alpha)}{\langle x,y\rangle\in [U], ~ t\subseteq x, ~ u\subseteq y} ~ \leq ~ \kappa^+.$$ Then $r(\emptyset,\emptyset,\alpha)=\kappa^+$ holds for all $\alpha<\kappa$, because the assumptions $(\kappa^+)^M=\kappa^+$ implies that for every $\gamma<\kappa^+$, there exists $x\in[T]^M\subseteq[T]=p[U]$ with $\mathrm{rnk}_x(\alpha)\geq\gamma$ and hence there exists $y\in{}^\kappa\kappa$ such that the pair $\langle x,y\rangle$ witnesses that $r(\emptyset,\emptyset,\alpha)\geq\gamma$. 

 \begin{claim*}
  If $\gamma<\kappa^+$, $(t,u)\in U$ and $\alpha<\kappa$ with $r(t,u,\alpha)=\kappa^+$, then there is $(v,w)\in U$ and $\alpha<\beta<\dom{v}$ such that $t\subsetneq v$, $u\subsetneq w$, $\dom{v}$ is closed under $\goedel{\cdot}{\cdot}$, $\beta<_v\alpha$ and $r(v,w,\beta)\geq\gamma$.  
 \end{claim*}

 \begin{proof}[Proof of the Claim]
  By our assumptions, we can find $\langle x,y\rangle\in U$ with $t\subseteq x$, $u\subseteq y$ and $\mathrm{rnk}_x(\alpha)\geq\gamma+\kappa$. Then there is $\alpha<\beta<\kappa$ with $\beta<_x\alpha$ and $\mathrm{rnk}_x(\beta)\geq\gamma$. 
  Pick $\xi>\beta+\dom{t}$ closed under $\goedel{\cdot}{\cdot}$. Then $t\subseteq x\restriction\xi$, $u\subsetneq y\restriction\xi$, $\beta<_{x\restriction\xi}\alpha$ and $r(x\restriction\xi,y\restriction\xi,\beta)\geq\gamma$.  
 \end{proof}

 \begin{claim*}
  If $\langle t,u\rangle\in U$ and $\alpha<\kappa$ with $r(t,u,\alpha)=\kappa^+$, then there exists $(v,w)\in U$ and $\alpha<\beta<\dom{v}$ such that $t\subsetneq v$, $u\subsetneq w$, $\dom{v}$ is closed under $\goedel{\cdot}{\cdot}$, $\beta<_v\alpha$ and $r(v,w,\beta)=\kappa^+$.  
 \end{claim*}

  \begin{proof}[Proof of the Claim]
   For each $\gamma<\kappa^+$, let $\langle v_\gamma,w_\gamma\rangle\in U$ and $\alpha<\beta_\gamma<\dom{v_\gamma}$ be the objects given by the above claim. 
   Since $\kappa=\kappa^{{<}\kappa}$ holds, we can find $\langle v,w\rangle\in U$, $\beta<\kappa$ and $X\subseteq\kappa^+$ with $\betrag{X}=\kappa^+$ such that $v_\gamma=v$, $w_\gamma=w$ and $\beta_\gamma=\beta$ for all $\gamma\in X$. But this implies that $r(v,w,\beta)=\kappa^+$. 
  \end{proof}

 Using the last claim, we now construct sequences $\seq{\langle t_n,u_n\rangle\in U}{n<\omega}$ and $\seq{\alpha_n<\kappa}{n<\omega}$ such that $\dom{t_{n+1}}$ is closed under $\goedel{\cdot}{\cdot}$, $t_n\subsetneq t_{n+1}$, $u_n\subsetneq u_{n+1}$, $\alpha_n<\alpha_{n+1}<\dom{t_{n+1}}$ and $\alpha_{n+1}<_{t_{n+1}}\alpha_n$ for all $n<\omega$. Set $t=\bigcup_{n<\omega}t_n$ and $u=\bigcup_{n<\omega}u_n$. Then the properties of $U$ imply that $\langle t,u\rangle\in U$ and there exists $x\in[T]$ with $t\subseteq x$. But then $(\kappa,<_x)$ is a well-order with $\alpha_{n+1}<_x\alpha_n$ for all $n<\omega$, a contradiction. 
\end{proof}

 As promised above, we end this section by showing that, for certain Kurepa trees, the converse of the statement of Corollary \ref{corollary:ContImageCantosSubtree} does not hold true.

 \begin{corollary}
  Let $\mu$ be an uncountable regular cardinal, let $\theta>\kappa$ be inaccessible cardinals above $\mu$, let $G$ be $\Coll{\kappa}{{<}\theta}$-generic over $\VV$ and let $H$ be $\Coll{\mu}{{<}\kappa}$-generic over $\VV[G]$. Then the following statements hold in $\VV[G,H]$: 
 \begin{enumerate}
  \item There is a $\kappa$-Kurepa subtree $T$ of ${}^{{<}\kappa}\kappa$ with the property that the set $[T]$ is not a continuous image of ${}^\kappa\kappa$. 

  \item Every $\kappa$-Kurepa tree contains a Cantor subtree. 
 \end{enumerate}
 \end{corollary}

 \begin{proof}
  First, since $(\kappa^+)^{\VV[G,H]}=\theta=(\kappa^+)^{\VV[G]}$ and $\kappa$ is inaccessible in $\VV[G]$, we can apply Theorem \ref{theorem:KurepaInaccessibleWO} to conclude that, in $\VV[G,H]$, there is a $\kappa$-Kurepa subtree $T$ of ${}^{{<}\kappa}\kappa$ such that the set $[T]$ is not a continuous image of ${}^\kappa\kappa$. 
  Next, fix a $\Coll{\mu}{{<}\kappa}$-nice name $\dot{T}\in\VV[G]$ for a $\kappa$-Kurepa tree with underlying set $\kappa$. Since $\Coll{\mu}{{<}\kappa}$ satisfies the $\kappa$-chain condition in $\VV[G]$ and $\Coll{\kappa}{{<}\theta}$ satisfies the $\theta$-chain condition in $\VV$, we can find $\kappa<\vartheta<\theta$ with the property that $\dot{T}\in\VV[G_\vartheta]$, where $G_\vartheta=G\cap\Coll{\kappa}{{<}\vartheta}$. Then $\dot{T}^H\in\VV[G_\vartheta,H]$, $\theta$ is inaccessible in $\VV[G_\vartheta,H]$ and hence forcing with $\Coll{\kappa}{[\vartheta,\theta)}^\VV$ over $\VV[G_\vartheta,H]$ adds a new cofinal branch to $\dot{T}^H$. Since the partial order $\Coll{\kappa}{[\vartheta,\theta)}^\VV$ is ${<}\mu$-closed in $\VV[G_\vartheta,H]$ and $\mu$ is an uncountable cardinal in $\VV[G_\vartheta,H]$, standard arguments show that $\dot{T}^H$ contains a Cantor subtree in $\VV[G_\vartheta,H]$ and this subtree is still a subtree in $\VV[G,H]$.  
 \end{proof}


\section{Superthin Kurepa trees}

The concepts introduced in the next definition will play a central role in our proofs of Theorems  \ref{theorem:KurepaImageInL},  \ref{theorem:CollapseImageNONImage} and  \ref{theorem:StatSlimImage}.

\begin{definition}
 Let $\kappa$ be an infinite regular cardinal and let $T$ be a subtree of ${}^{{<}\kappa}\kappa$. 
 \begin{enumerate}
  \item The tree $T$ is \emph{pruned} if for every $s\in T$, there is $t\in T$ with $s\subsetneq t$. 

  \item The \emph{boundary} $\partial T$ of $T$ is defined as the set of minimal elements of ${}^{<\kappa}\kappa\setminus T$, i.e. $$\partial T ~ = ~ \Set{t\in{}^{{<}\kappa}\kappa\setminus T}{\forall\alpha\in\dom{t} ~ t\restriction\alpha\in T}.$$


  \item The tree $T$ is \emph{superthin} if $\betrag{(T\cup\partial T)\cap{}^\alpha\kappa}<\kappa$ holds for all $\alpha\in\Lim\cap\kappa$.  
 \end{enumerate}
\end{definition}

 
Note that a subtree $T$ of ${}^{{<}\kappa}\kappa$ of height $\kappa$ is ${<}\kappa$-closed if and only if 
$\partial T\cap{}^\alpha\kappa=\emptyset$ holds for all $\alpha\in\Lim\cap\kappa$. 
The following observation shows that the fact that all non-empty closed subsets of ${}^\omega\omega$ are retracts of ${}^\omega\omega$ can be generalized to a certain class of closed subsets of higher Baire spaces.

\begin{proposition}\label{proposition:PrunedClosedImage}
  If $\kappa$ is an infinite regular cardinal and $T$ is a ${<}\kappa$-closed pruned subtree of ${}^{{<}\kappa}\kappa$, then $N_t\cap[T]\neq\emptyset$ for every $t\in T$ and the closed set $[T]$ is a retract of ${}^\kappa\kappa$. 
\end{proposition}

\begin{proof}
 Given $t\in T$, our assumptions on $T$ allow us to do an easy inductive construction that produces a sequence $\seq{t_\alpha}{\alpha<\kappa}$ of elements of $T$ such that $t_0=t$, $\dom{t_\alpha}=\dom{t}+\alpha$ and $t_\alpha\subseteq t_\beta$ for all $\alpha\leq\beta<\kappa$. In this situation, we have $x_t=\bigcup\Set{t_\alpha}{\alpha<\kappa}\in N_t\cap [T]$. 

 Now, fix $y\in{}^\kappa\kappa$. Then there is a  unique $\beta<\kappa$ with $y\restriction\beta\in\partial T$ and, since $T$ is ${<}\kappa$-closed, we know that $\beta$ is not a limit ordinal. Hence there is a unique $\alpha_y<\kappa$ with $y\restriction\alpha_y\in T$ and $y\restriction(\alpha_y+1)\notin T$. 

 Let $\map{r}{{}^\kappa\kappa}{[T]}$ denote the unique function with $r\restriction[T]=\id_{[T]}$ and $r(y)=x_{y\restriction\alpha_y}$ for all $y\in{}^\kappa\kappa\setminus[T]$. 

 \begin{claim*}
  The function $r$ is continuous. 
 \end{claim*}

 \begin{proof}[Proof of the Claim]
  First, fix $x\in[T]$, $\alpha<\kappa$ and $y\in N_{x\restriction\alpha}$. If $y\in[T]$, then we have $r(y)=y\in N_{x\restriction\alpha}=N_{r(x)\restriction\alpha}$. In the other case, if $y\notin[T]$, then $x\restriction\alpha\in T$ implies that $\alpha_y\geq\alpha$ and hence $r(y)=x_{y\restriction\alpha_y}\in N_{x\restriction\alpha}=N_{r(x)\restriction\alpha}$. 

  Now, fix $x\in{}^\kappa\kappa\setminus[T]$ and $\alpha<\kappa$. Pick $\alpha\leq\beta<\kappa$ with $x\restriction\beta\notin T$ and $y\in N_{x\restriction\beta}$. Then we have $y\notin[T]$, $\alpha_x=\alpha_y\leq\beta$, $x\restriction\alpha_x=y\restriction\alpha_y$ and  $r(y)=x_{y\restriction\alpha_y}=x_{x\restriction\alpha_x}=r(x)\in N_{r(x)\restriction\alpha}$. 
 \end{proof}

 Since $r\restriction[T]=\id_{[T]}$, the above claim completes the proof of the proposition. 
\end{proof}

\begin{lemma}\label{lemma:SuperthinContImage}
 Let $\kappa$ be an uncountable regular cardinal. If there is a superthin $\kappa$-Kurepa subtree $S$ of ${}^{{<}\kappa}\kappa$, then there is a $\kappa$-Kurepa subtree $T$ of ${}^{{<}\kappa}\kappa$ with the property that $[T]$ is a retract of ${}^\kappa\kappa$. 
 Moreover, if $S$ is pruned, then $T$ can be taken to contain $S$ as a subtree. 
\end{lemma}

\begin{proof}
 Let $S_0$ be a superthin $\kappa$-Kurepa subtree of ${}^{{<}\kappa}\kappa$. If $S_0$ is pruned, then we set $S_1=S_0$. Otherwise, we define  $$S_1 ~ = ~ \Set{t\in S_0}{\exists x\in[T] ~ t\subseteq x}.$$ 
 Then it is easy to see that $S_1$ is a pruned superthin $\kappa$-Kurepa subtree of ${}^{{<}\kappa}\kappa$. Finally, define  $T$ to consist of all elements of $S_1$ together with all elements $t$ of ${}^{{<}\kappa}\kappa$ with the property that there exist $s\in\partial S_1$ such that $s\subseteq t$, $\dom{s}\in\Lim$ and $t(\alpha)=0$ for all $\alpha\in\dom{t}\setminus\dom{s}$. 

 \begin{claim*}
  $T$ is a ${<}\kappa$-closed pruned $\kappa$-Kurepa subtree of ${}^{{<}\kappa}\kappa$. 
 \end{claim*}

 \begin{proof}[Proof of the Claim]
  Note that for every $t\in T\setminus S_1$, there is a unique limit ordinal $\alpha\leq\dom{t}$ with $t\restriction\alpha\in\partial S_1$. Since $S_1$ is superthin, this shows that $\betrag{(T\setminus S_1)\cap{}^\alpha\kappa}<\kappa$ holds for all $\alpha<\kappa$. Moreover, since $S_1$ is a $\kappa$-Kurepa subtree of ${}^{{<}\kappa}\kappa$, this directly implies that $T$ is also a $\kappa$-Kurepa subtree of ${}^{{<}\kappa}\kappa$. Next, note that for every $t\in T\setminus S_1$, we have $t\subsetneq t\cup\{\langle\dom{t},0\rangle\}\in T$. Since $S_1$ is pruned, this shows that $T$ is also pruned. Finally, assume that there is $t\in\partial T$ with $\dom{t}\in\Lim$. Then $t\notin S_1$ and, by the definition of $T$, we have $t\restriction\alpha\in S_1$ for all $\alpha<\dom{t}$. But then $t\in\partial S_1\subseteq T$, a contradiction. 
 \end{proof}

 By the above claim, Proposition \ref{proposition:PrunedClosedImage} directly shows that $[T]$ is a retract of ${}^\kappa\kappa$. 
\end{proof}

The properties of tree introduced in the next definition are studied in depth by Bernhard K\"onig in \cite{MR2013395}.

\begin{definition}
 \begin{enumerate}
  \item A tree of height $\lambda$ is \emph{trivially coherent} if it isomorphic to a subtree of ${}^{{<}\lambda}2$ consisting of sequences $t$ with the property that $t^{{-}1}\{0\}$ is a finite set. 

  \item A tree $T$ is \emph{locally coherent} if $T_{{<}\alpha}$ is trivially coherent for every $\alpha<\height{T}$. 
 \end{enumerate}
\end{definition}

\begin{proposition}\label{proposition:LocallyCoherentSuperthin}
 Let $\kappa$ be a regular cardinal with $\lambda^\omega<\kappa$ for all $\lambda<\kappa$. Then every  locally coherent $\kappa$-Kurepa subtree of ${}^{{<}\kappa}\kappa$ is superthin. 
\end{proposition}

\begin{proof}
 Let $T$ be a locally coherent $\kappa$-Kurepa subtree of ${}^{{<}\kappa}\kappa$ and let $\alpha\in\Lim\cap\kappa$. 

 First, assume that $\cof{\alpha}=\omega$. Pick a  cofinal sequence $\seq{\alpha_n}{n<\omega}$ in $\alpha$. Given $n<\omega$, the fact that $T$ is a $\kappa$-Kurepa tree implies that $\lambda_n=\betrag{T\cap{}^{\alpha_n}\kappa}<\kappa$. Moreover, since $\kappa$ is regular, we know that $\lambda=\sup_{n<\omega}\lambda_n<\kappa$. But then the set $[T\cap{}^\alpha\kappa]$ has cardinality at most $\lambda^\omega<\kappa$ and hence we can conclude that $\betrag{(T\cup\partial T)\cap{}^\alpha\kappa}<\kappa$. 

 Next, assume that $\cof{\alpha}>\omega$. Fix a tree monomorphism $\map{\pi}{T\cap{}^{{<}\alpha}\kappa}{{}^{{<}\alpha}2}$ with the property that $\pi(t)^{{-}1}\{0\}$ is finite for every $t\in T\cap{}^{{<}\alpha}\kappa$. Given $u\in[T\cap{}^\alpha\kappa]$, we can now find minimal $\alpha_u<\alpha$ and $N_u<\omega$ with the property that $\betrag{\pi(u\restriction\bar{\alpha})^{{-}1}\{0\}}=N_u$ holds for all $\alpha_u\leq\bar{\alpha}<\alpha$. In this situation, two elements $u_0$ and $u_1$ of $[T\cap{}^\alpha\kappa]$ are identical if and only if 
$u_0\restriction\alpha_{u_0}=u_1\restriction\alpha_{u_1}$ 
holds. In particular, we also have  $\betrag{(T\cup\partial T)\cap{}^\alpha\kappa}\leq\betrag{[T\cap{}^\alpha\kappa]}<\kappa$ in this case. 
\end{proof}

 The following unpublished result of Donder shows that, in the constructible universe, locally coherent Kurepa trees exist at all successor cardinals. This result is  proven by showing that the initial segments of the canonical Kurepa trees constructed from the canonical morasses at successor cardinals in $\LL$ (see {\cite[Chapter VIII, Section 2]{MR750828}} and {\cite[Section 2]{MR701126}}) satisfy the criterion for trivial  coherency given by {\cite[Lemma 2.17]{MR2013395}}.

\begin{theorem}[Donder]\label{theorem:DonderKurepa}
 Assume that $\VV=\LL$ holds. If $\kappa$ is the successor of an infinite cardinal, then there is a locally coherent $\kappa$-Kurepa tree. 
\end{theorem}

\begin{proof}[Proof of Theorem \ref{theorem:KurepaImageInL}]
 Assume that $\VV=\LL$ and let $\kappa$ be an uncountable regular cardinal. 
 If there is a cardinal $\mu$ with $\kappa=\mu^+$ and $\cof{\mu}=\omega$, then Theorem \ref{theorem:Main-Negative} directly implies that there is no $\kappa$-Kurepa subtree $T$ of ${}^{{<}\kappa}\kappa$ with the property that $[T]$ is a continuous image of ${}^\kappa\kappa$. 
 Moreover, if $\kappa$ is inaccessible, then ${}^{{<}\kappa}2$ is a $\kappa$-Kurepa subtree of ${}^{{<}\kappa}\kappa$ and the set $[{}^{{<}\kappa}2]={}^\kappa 2$ is a retract of $\kappa$. 
 Finally, if $\kappa$ is the successor of a cardinal of uncountable cofinality, then $\lambda^\omega<\kappa$ holds for all $\lambda<\kappa$ and hence a combination of Lemma \ref{lemma:SuperthinContImage}, Proposition \ref{proposition:LocallyCoherentSuperthin} and Theorem \ref{theorem:DonderKurepa} shows that there is a $\kappa$-Kurepa subtree $T$ of ${}^{{<}\kappa}\kappa$ with the property that $[T]$ is a retract of ${}^\kappa\kappa$. 
 Since the $\GCH$ holds in $\LL$, the above observations provide the desired characterization. 
\end{proof}

Our proofs of Theorem \ref{theorem:CollapseImageNONImage} and Theorem \ref{theorem:StatSlimImage} will heavily rely on the following properties of inner models that were isolated by Hamkins in \cite{MR2063629}.

\begin{definition}
 Let $\mu$ be an infinite cardinal  and let $M$ be an inner model.
 \begin{enumerate}
  \item The pair $(M,\VV)$ has the \emph{$\mu$-cover property} if for every set $x$ with $x\subseteq M$ and $\betrag{x}<\mu$, there is $c\in M$ with $x\subseteq c$ and $\betrag{c}^M<\mu$. 
 
  \item The pair $(M,\VV)$ has the \emph{$\mu$-approximation property} if $x\in M$ holds for every set $x$ with $x\subseteq M$ and $a\cap x\in M$ whenever $a\in M$ with $\betrag{a}^M<\mu$. 
 \end{enumerate}
\end{definition}

\begin{lemma}\label{lemma:SuperthinFromInnerModel}
 Let $M$ be an inner model such that the following statements hold for infinite regular cardinals $\mu<\kappa$ with $\lambda^{{<}\mu}<\kappa$ for all $\lambda<\kappa$: 
 \begin{enumerate}
  \item The pair $(M,\VV)$ satisfies the $\mu$-approximation property. 

  \item $\kappa$ is inaccessible in $M$. 

  \item $(2^\kappa)^M\geq\kappa^+$. 
 \end{enumerate}
 Then $T=({}^{{<}\kappa}2)^M$ is a pruned superthin $\kappa$-Kurepa subtree of ${}^{{<}\kappa}\kappa$. 
\end{lemma}

\begin{proof}
 First, note that our second and third assumption directly imply that $T$ is a pruned $\kappa$-Kurepa subtree of ${}^{{<}\kappa}\kappa$. 
 Fix $\alpha\in\Lim\cap\kappa$. If $\cof{\alpha}<\mu$, then our assumptions directly imply that $\betrag{\partial T\cap{}^\alpha\kappa}\leq\betrag{[T\cap{}^{{<}\alpha}\kappa]}<\kappa$. 
 In the other case, if $\cof{\alpha}\geq\mu$ and $x\in[T\cap{}^{{<}\alpha}\kappa]$, then $a\cap x\in M$ holds for all $a\in M$ with $\betrag{a}^M<\mu$ and hence the $\mu$-approximation property implies that $x$ is an element of $M$. This argument shows that $[T\cap{}^{{<}\alpha}\kappa]\subseteq T$ holds for all $\alpha\in\Lim\cap\kappa$ with $\cof{\alpha}\geq\mu$. In particular, this shows that $\partial T\cap{}^\alpha\kappa=\emptyset$ holds for all such ordinals $\alpha$. 
 Since $T$ is a $\kappa$-Kurepa subtree of ${}^{{<}\kappa}\kappa$, these computations show that $T$ is superthin.  
\end{proof}

The proof of {\cite[Theorem 7.2]{luecke-schlicht-continuous-images}} directly yields the following result needed for the proof of Theorem \ref{theorem:CollapseImageNONImage}.

\begin{theorem}\label{theorem:CoverPropNotImage}
 Let $M$ be an inner model such that $\RRR\nsubseteq M$ and the pair $(M,\VV)$ has the $\aleph_1$-cover property. If $\kappa$ is an uncountable regular cardinal with $\betrag{(2^\kappa)^M}>\kappa$ and $T=({}^{{<}\kappa}2)^M$, then the set $[T]$ is not a continuous image of ${}^\kappa\kappa$. 
\end{theorem}

\begin{proof}[Proof of Theorem \ref{theorem:CollapseImageNONImage}]
 Let $\kappa$ be an inaccessible cardinal, let $\mu<\kappa$ be an uncountable regular cardinal, let $x$ be $\Add{\omega}{1}$-generic over $\VV$ and let $G$ be $\Coll{\mu}{{<}\kappa}$-generic over $\VV[x]$. Then all cofinalities less than or equal to $\mu$ and greater than or equal to $\kappa$ are preserved in $\VV[x,G]$ and $\kappa=(\mu^+)^{\VV[x,G]}$. Moreover, {\cite[Lemma 13]{MR2063629}} shows that the pair $(\VV,\VV[x,G])$ has the $\aleph_1$-approximation and $\aleph_1$-cover property. 
Set $T_1=({}^{{<}\kappa}2)^\VV$. Since $\kappa$ is inaccessible in $\VV$, $(2^\kappa)^\VV=(2^\kappa)^{\VV[x,G]}$ and $(\lambda^\omega)^{\VV[x,G]}=(\lambda^\omega)^\VV<\kappa$ holds for all $\lambda<\kappa$, Lemma  \ref{lemma:SuperthinFromInnerModel} shows that $T_1$ is  a pruned superthin $\kappa$-Kurepa subtree of ${}^{{<}\kappa}\kappa$. 
 In addition, since $\RRR^{\VV[x,G]}\nsubseteq\VV$, we can apply Theorem \ref{theorem:CoverPropNotImage} to conclude that the set $[T_1]$ is not a continuous image of ${}^\kappa\kappa$ in $\VV[x,G]$.  
 Finally, an application of Lemma \ref{lemma:SuperthinContImage} yields a $\kappa$-Kurepa subtree $T_0$ of ${}^{{<}\kappa}\kappa$ in $\VV[x,G]$ such that $T_1\subseteq T_0$ and the set $[T_0]$ is a retract of ${}^\kappa\kappa$ in $\VV[x,G]$.  
\end{proof}

\begin{proof}[Proof of Theorem \ref{theorem:StatSlimImage}]
 Let $\kappa$ be a $2$-Mahlo cardinal, let $E$ denote the set of Mahlo cardinals smaller than $\kappa$ and let $x$ be $\Add{\omega}{1}$-generic over $\VV$. Work in $\VV[x]$ and let $\langle\seq{\vec{\PPP}_{{<}\alpha}}{\alpha\leq\kappa}, ~ \seq{\dot{\PPP}_\alpha}{\alpha<\kappa}\rangle$ be a forcing iteration with Easton support such that the following statements hold whenever $H$ is $\vec{\PPP}_{{<}\alpha}$-generic over $\VV[x]$ for some $\alpha<\kappa$: 
  \begin{itemize}
   \item If $\alpha\in E$, then $\dot{\PPP}_\alpha^H=\Coll{\alpha}{2^\alpha}^{\VV[x,H]}$. 
   
   \item If $\alpha\notin E$, then $\dot{\PPP}_\alpha^H$ is the trivial partial order. 
  \end{itemize}
  
  \begin{claim*}
   Forcing with $\vec{\PPP}_{{<}\kappa}$ over $\VV[x]$ preserves the inaccessibility of $\kappa$, the stationarity of $E$ in $\kappa$ and the regularity of all elements of $E$ and of all regular cardinals greater than or equal to $\kappa$. 
  \end{claim*}
  
  \begin{proof}[Proof of the Claim]
   Let $\alpha\leq\kappa$ be a Mahlo cardinal, let $G$ be $\vec{\PPP}_{{<}\kappa}$-generic over $\VV[x]$ and let $H$ denote the filter on $\vec{\PPP}_{{<}\alpha}$ induced by $G$. 
   Then {\cite[Proposition 7.13]{MR2768691}} shows that $\vec{\PPP}_{{<}\alpha}$ satisfies the $\alpha$-chain condition in $\VV[x]$ and {\cite[Proposition 7.12]{MR2768691}} implies that the induced tail forcing $\dot{\PPP}_{[\alpha,\kappa)}^H$ is ${<}\alpha$-closed in $\VV[x,H]$. In particular, $\alpha$ is regular in $\VV[x,G]$, stationary subsets of $\alpha$ in $\VV$ are stationary in $\VV[x,G]$ and, if $\alpha<\kappa$, then $(2^{{<}\alpha})^{\VV[x,G]}=\alpha$. 
  \end{proof}
  
  Let $G$ be $\vec{\PPP}_{{<}\kappa}$-generic over $\VV[x]$. Since {\cite[Proposition 7.12]{MR2768691}} implies that $\vec{\PPP}_{{<}\kappa}$ is $\sigma$-closed in $\VV[x]$, we can apply {\cite[Lemma 13]{MR2063629}} to conclude that the pair $(\VV,\VV[x,G])$ has the $\aleph_1$-approximation property. Set $T_0=({}^{{<}\kappa}2)^\VV$ and work in $\VV[x,G]$. Then Lemma \ref{lemma:SuperthinFromInnerModel} implies that $T_0$ is a superthin $\kappa$-Kurepa subtree of ${}^{{<}\kappa}\kappa$. Let $T$ be the tree constructed from $T_0$ as in the proof of Lemma \ref{lemma:SuperthinContImage}. 
  
  \begin{claim*}
   If $\alpha\in E$, then $\betrag{T\cap{}^\alpha\kappa}=\alpha$. 
  \end{claim*}
  
  \begin{proof}[Proof of the Claim]
   First, note that, if we repeat the construction from the proof of Lemma \ref{lemma:SuperthinContImage}, then $S_0=S_1$. Fix $\alpha\in E$. Then our forcing construction ensures that $\betrag{T_0\cap{}^\alpha\kappa}\leq\betrag{(2^\alpha)^\VV}=\alpha$ holds. 
   Next, notice that, if $\bar{\alpha}\in\Lim\cap\alpha$, then $\partial T\cap{}^{\bar{\alpha}}\kappa\subseteq{}^{\bar{\alpha}}2$ and hence    the proof of the previous claim allows us to conclude that $\betrag{\partial T_0\cap{}^{\bar{\alpha}}\kappa}\leq2^{{<}\alpha}=\alpha$. 
   Finally, since the pair $(\VV,\VV[x,G])$ has the $\aleph_1$-approximation property and $\alpha$ is regular in $\VV[x,G]$, we also know that $\partial T_0\cap{}^\alpha\kappa=\emptyset$. By the definition of $T$, these observations imply the statement of the claim.    
  \end{proof}
  
  Since the proof of Lemma \ref{lemma:SuperthinContImage} shows that $T$ is a $\kappa$-Kurepa subtree of ${}^{{<}\kappa}\kappa$ and the set $[T]$ is a retract of ${}^\kappa\kappa$, the above claim shows that the tree possesses all of the desired properties. 
\end{proof}


\section{Isolated points}

The following simple observation will be central for our investigation of retractions of generalized Baire space onto the sets of cofinal branches of Kurepa trees.

\begin{proposition}\label{proposition:SplitPlusInclude}
 Let $\kappa$ be an uncountable regular cardinal and let $r$ be a retraction from ${}^\kappa\kappa$ to a subset $X$ of ${}^\kappa\kappa$. If $\alpha<\kappa$ and $A\in[X]^{{<}\kappa}$, then there is $\alpha<\beta<\kappa$ such that the following statements hold: 
 \begin{enumerate}
   \item $x\restriction\beta\neq y\restriction\beta$ for all $x,y\in A$ with $x\neq y$. 
   
   \item $r[N_{x\restriction\beta}]\subseteq N_{x\restriction\beta}$ for all $x,y\in A$. 
 \end{enumerate}
\end{proposition}

\begin{proof}
  We inductively construct a strictly increasing sequence $\seq{\beta_n}{n<\omega}$ of ordinals in the interval $(\alpha,\kappa)$. Since $\betrag{A}<\kappa$, we can find $\alpha<\beta_0<\kappa$ with $x\restriction\beta_0\neq y\restriction\beta_0$ for all $x,y\in A$ with $x\neq y$. Next, assume that $\beta_n$ is already defined for some $n<\omega$. Given $x\in A$, we can then find $\beta_n^x\in(\beta_n,\kappa)$ with $r[N_{x\restriction\beta_n^x}]\subseteq N_{x\restriction\beta_n}$. Set $\beta_{n+1}=\sup_{x\in A}\beta^x_n<\kappa$. Finally, define $\beta=\sup_{n<\omega}\beta_n<\kappa$. Given $x\in A$, this construction ensures that $r[N_{x\restriction\beta_{n+1}}]\subseteq N_{x\restriction\beta_n}$ holds for all $n<\omega$ and this allows us to conclude that $r[N_{x\restriction\beta}]\subseteq N_{x\restriction\beta}$ holds. 
\end{proof}

Using the above proposition, we now show that Kurepa trees that are retracts contain isolated cofinal branches.

\begin{proof}[Proof of Theorem \ref{theorem:ResultsKurepaNoIsolated}] 
 Let $\kappa$ be an uncountable regular cardinal and let $T$ be a $\kappa$-Kurepa subtree of ${}^{{<}\kappa}\kappa$ such that  
  there is a continuous function 
 $\map{r}{{}^\kappa\kappa}{[T]}$ with $r\restriction[T]=\id_{[T]}$ and $[T]$ does not have isolated points. Let $\mu\leq\kappa$ be the least cardinal with $2^\mu\geq\kappa$. 
 
  In the following, we inductively construct a strictly increasing continuous sequence $\seq{\alpha(\xi)<\kappa}{\xi\leq\mu}$ of ordinals and a sequence $\seq{t_s\in T\cap{}^{\alpha(\xi)}\kappa)}{\xi\leq\mu, ~ s\in{}^\xi 2}$ such that the following statements hold for all $\xi\leq\mu$ and $s_0,s_1\in {}^\xi 2$: 
 \begin{enumerate}
  \item If $\zeta<\xi$, then $t_{s_0\restriction\zeta}=t_{s_0}\restriction\alpha(\zeta)$. 

  \item For all $\zeta\leq\xi$, we have $t_{s_0}\restriction\alpha(\zeta)=t_{s_1}\restriction\alpha(\zeta)$ if and only if $s_0\restriction\zeta=s_1\restriction\zeta$. 

  \item If $\xi\in\Lim$, then $t_{s_0}=\bigcup_{\zeta<\xi}t_{s_0\restriction\zeta}$. 

  \item $r[N_{t_{s_0}}]\subseteq [T]\cap N_{t_{s_0}}$. 
 \end{enumerate}

 Set $\alpha(0)=0$ and $t_\emptyset=\emptyset$. 
  Next, assume that $\xi<\mu$ and the sequences $\seq{\alpha(\zeta)}{\zeta\leq\xi}$ and $\seq{t_s}{\zeta\leq\xi, ~ s\in{}^\zeta 2}$ are already constructed.  Fix a sequence $s$ in ${}^\xi 2$. 
 Since we have $r[N_{t_s}]\subseteq [T]\cap N_{t_s}\neq\emptyset$ and $[T]$ contains no isolated points, there are $x^s_0,x^s_1\in N_{t_s}\cap[T]$ with $x^s_0\neq x^s_1$. 
 Set $A_\xi=\Set{x^s_i}{s\in{}^\xi 2, ~ i<2}$. Then the minimality of $\mu$ implies that $\betrag{A_\xi}<\kappa$ and hence we can apply Proposition \ref{proposition:SplitPlusInclude} to $\alpha(\xi)$ and $A_\xi$ to find $\alpha(\xi)<\alpha(\xi+1)<\kappa$ with the listed properties. 
  Given $s\in{}^\xi 2$ and $i<2$, we can then define $t_{s^\frown\langle i\rangle}=x^s_i\restriction\alpha(\xi+1)$. 
 Finally, let $\xi\leq\mu$ be a limit ordinal. Set $\alpha(\xi)=\sup_{\zeta<\xi}\alpha(\zeta)$ and define  $t_s=\bigcup_{\zeta<\xi}t_{s\restriction\zeta}$ for all $s\in{}^\xi 2$. Given  $s\in{}^\xi 2$, we then have 
 $$r[N_{t_s}] ~ \subseteq ~ \bigcap_{\zeta<\xi} r[N_{t_{s\restriction\zeta}}] ~ \subseteq ~ \bigcap_{\zeta<\xi} ([T]\cap N_{t_{s\restriction\zeta}})  
  ~ = ~ [T]\cap N_{t_s}$$ and, since $\emptyset\neq r[N_{t_s}]\subseteq [T]\cap N_{t_s}$, this shows that $t_s\in T$. 
  
  Now, assume that $\kappa$ is not inaccessible. Then $\mu<\kappa$ and the above construction shows that $$\kappa ~ \leq ~ 2^\mu ~ = ~ \betrag{\Set{t_s}{s\in{}^\mu 2}} ~ \leq ~ \betrag{T\cap{}^{\alpha(\mu)}\kappa},$$ a contradiction. 
  
  This shows that $\kappa$ is inaccessible and $\kappa=\mu$. Then the set $$C ~ = ~ \Set{\xi<\kappa}{\textit{$\alpha(\xi)=\xi$ is a cardinal}}$$ is closed and unbounded in $\kappa$. Given $\xi\in C$, the above construction ensures that the set $T\cap{}^\xi\kappa$ has cardinality $2^\xi>\xi$. In particular, the tree $T$ is not stationary slim. 
\end{proof}

Our next goal is to show that every Kurepa tree can be thinned out to obtain a Kurepa tree without isolated branches. The proof of the next lemma is a direct adaptation of a classical argument to higher cardinals.

\begin{lemma}\label{lemma:KurepaNoIsolated}
 Let $\kappa$ be an uncountable regular cardinal and let $T$ be a subtree of ${}^{{<}\kappa}\kappa$ with $\betrag{[T]}>\kappa^{{<}\kappa}$.  Then there is a subtree $S$ of ${}^{{<}\kappa}\kappa$ such that $S\subseteq T$, $\betrag{[T]}=\betrag{[S]}$ and $[S]$ contains no isolated points. 
\end{lemma}

\begin{proof}
  Set $\theta=\kappa^{{<}\kappa}$ and let $\seq{C_\gamma}{\gamma\leq \theta^+}$ denote the unique sequence of closed subsets of ${}^\kappa\kappa$ such that $C_0=[T]$ and the following statements hold for all $\gamma\leq\theta^+$: 
 \begin{enumerate}
  \item If $\gamma<\theta^+$ and $I_\gamma$ is the set of isolated points of $C_\gamma$, then $C_{\gamma+1}=C_\gamma\setminus I_\gamma$. 

  \item If $\gamma$ is a limit ordinal, then $C_\gamma=\bigcap_{\beta<\gamma}C_\beta$. 
 \end{enumerate}
 
 Assume, towards a contradiction, that $C_\gamma\neq C_{\theta^+}$ for every $\gamma<\theta^+$. Then $I_\gamma\neq\emptyset$ for every $\gamma<\theta^+$. 
 Given $\gamma<\theta^+$ and $x\in I_\gamma$, let $\alpha^x_\gamma$ denote the least $\alpha<\kappa$ with $C_\gamma\cap N_{x\restriction\alpha}=\{x\}$.
  Next, define 
 $K_\gamma=\Set{x\restriction \alpha^x_\gamma\in{}^{{<}\kappa}\kappa}{x\in I_\gamma}\neq\emptyset$. 
 Then we have $K_\gamma\cap K_{\bar{\gamma}}=\emptyset$ for all $\bar{\gamma}<\gamma<\theta^+=(\kappa^{{<}\kappa})^+$, a contradiction. 

 The above computations show that there is a $\gamma<\theta^+$ with $C_\gamma=C_{\theta^+}$ and this shows that the closed set $C_\gamma$ has no isolated points. 
 Let $S$ denote the canonical subtree of ${}^{{<}\kappa}\kappa$ with $[S]=C_\gamma$. 
 Then $S\subseteq T$ and $$\betrag{[S]} ~ = ~ \betrag{[T]\setminus\bigcup_{\beta<\gamma}I_\beta} ~ = ~ \betrag{[T]},$$ because  for every $\beta<\theta^+$, the fact that $I_\beta$ consists of the isolated points of $C_\beta$, allows us to conclude that $\betrag{I_\beta}\leq\kappa^{{<}\kappa}<\betrag{[T]}$. 
\end{proof}

Using the above lemma, we are now able to prove the remaining results from Section \ref{subsection:NoRetracts}.

\begin{proof}[Proof of Theorem \ref{theorem:NotAllRetracts}]
 Let $\kappa$ be an uncountable regular cardinal with $\kappa^{<\kappa}=\kappa$ and assume there exists a $\kappa$-Kurepa tree. If $\kappa$ is inaccessible, then Theorem \ref{theorem:KurepaInaccessibleWO} shows that there is a $\kappa$-Kurepa subtree $T$ of ${}^{{<}\kappa}\kappa$ with the property that the set $[T]$ is not a retract of ${}^\kappa\kappa$. In the other case, if $\kappa$ is not inaccessible, then our assumption allows us to use Lemma \ref{lemma:KurepaNoIsolated} to find a $\kappa$-Kurepa subtree $T$ of ${}^{{<}\kappa}\kappa$ with the property that the set $[T]$ has no isolated points and Theorem \ref{theorem:ResultsKurepaNoIsolated} then shows that the set $[T]$ is not a retract of ${}^\kappa\kappa$. 
 Finally, if there exists a stationary slim $\kappa$-Kurepa subtree $T$ of ${}^{{<}\kappa}\kappa$ and $S$ is the subtree of $T$ produced by an application of Lemma \ref{lemma:KurepaNoIsolated}, then $S$ is also stationary slim, the set $[S]$ has no isolated points and hence Theorem \ref{theorem:ResultsKurepaNoIsolated} shows that the set $[S]$ is not a retract of ${}^\kappa\kappa$. 
\end{proof}

\begin{proof}[Proof of Theorem \ref{theorem:ContImageButNotRetracts}]
 Let $\kappa$ be an uncountable regular cardinal, let $S$ be a $\kappa$-Kurepa subtree of ${}^{{<}\kappa}\kappa$ and let $\map{c}{{}^\kappa\kappa}{[S]}$ be a continuous surjection. We let $I$ denote the set of isolated points of $[S]$ and we fix an injection $\map{\iota}{I}{S}$ with $N_{\iota(z)}\cap[S]=\{z\}$ for all $z\in I$. Define $T$ to be the union of $S$ and the set $$\Set{t\in{}^{{<}\kappa}\kappa}{\exists z\in I ~ [\iota(z)\subseteq t\wedge \betrag{\Set{\alpha\in\dom{t}\setminus\dom{\iota(z)}}{t(\alpha)\neq z(\alpha)}}<\omega]}.$$ 
 Then $T$ is a subtree of ${}^{{<}\kappa}\kappa$ with $\betrag{T\cap{}^\alpha\kappa}\leq\betrag{S\cap{}^\alpha\kappa}+\betrag{\alpha}+\omega<\kappa$ for all $\alpha<\kappa$. Moreover,  it is easy to see that $$[T] ~ = ~ [S] ~ \cup ~ \Set{y\in{}^\kappa\kappa}{\exists z\in I ~ [\iota(z)\subseteq y\wedge \betrag{\Set{\alpha<\kappa}{y(\alpha)\neq z(\alpha)}}<\omega]}.$$ This shows that  $T$ is a $\kappa$-Kurepa subtree of ${}^{{<}\kappa}\kappa$ with the property that the set $[T]$ does not contain isolated points and, if $S$ is stationary slim, then $T$ is also stationary slim. 
 
 For each $z\in I$, the set $N_{\iota(z)}\cap[T]$ has cardinality $\kappa$ and we can fix an enumeration $\seq{y_z(\alpha)}{\alpha<\kappa}$ of this set. Moreover, we pick an injection $\map{\rho}{I}{{}^{{<}\kappa}\kappa}$ with the property that  $c[N_{\rho(z)}]\subseteq N_{\iota(z)}$ holds for all $z\in I$. Then $c[N_{\rho(z)}]=\{z\}$ holds for all $z\in I$. Set $O=\bigcup\Set{N_{\rho(z)}}{z\in I}$ and let $\map{d}{{}^\kappa\kappa}{[T]}$ denote the unique function with $d\restriction({}^\kappa\kappa\setminus O)=c\restriction({}^\kappa\kappa\setminus O)$ and $d(x)=y_{c(x)}(x(\dom{\rho(c(x))}))$ for all $x\in O$.  Then $d$ is a surjection. 
 
 \begin{claim*}
  The map $d$ is continuous. 
 \end{claim*}
 
 \begin{proof}[Proof of the Claim]
  Fix $x\in{}^\kappa\kappa$ and $\beta<\kappa$. 

  First, assume that $x\in O$. Then $d[N_{x\restriction(\dom{\rho(c(x))}+1)}]=\{d(x)\}\subseteq N_{d(x)\restriction\beta}$. 

  Next, assume that $x\notin O$ and $c(x)\in I$. Since $c(x)$ is isolated in $T$ and $x\notin O$, there is $\alpha<\kappa$ with $N_{x\restriction\alpha}\cap N_{\rho(c(x))}=\emptyset$ and $c[N_{x\restriction\alpha}]=\{c(x)\}$. Then $N_{x\restriction\alpha}\cap O=\emptyset$ and we can conclude that $d[N_{x\restriction\alpha}]=c[N_{x\restriction\alpha}]=\{c(x)\}\subseteq N_{d(x)\restriction\beta}$. 

  Finally, assume that $x\notin O$ and $c(x)\notin I$. 
  Then there is $\alpha<\kappa$ with $c[N_{x\restriction\alpha}]\subseteq N_{c(x)\restriction\beta}$. 
 Fix $u\in N_{x\restriction\alpha}\cap O$. 
 Then $c(u)\in N_{\iota(c(u))}\cap N_{c(x)\restriction\beta}\neq\emptyset$ and this implies that the sequences $\iota(c(u))$ and $c(x)\restriction\beta$ are comparable.  
 But then $c(x)\restriction\beta\subsetneq\iota(c(u))$, because $\iota(c(u))\subseteq c(x)\restriction\beta$ would imply that $c(x)\in N_{c(x)\restriction\beta}\cap[T]\subseteq N_{\iota(\rho(c(u)))}\cap[T]=\{c(u)\}\subseteq I$. 
 This shows that $d(u)=y_{c(u)}(u(\dom{\rho(c(u))}))\in N_{\iota(c(u))}\subseteq N_{c(x)\restriction\beta}$. Since $c(x)=d(x)$, these computations show that $d[N_{x\restriction\alpha}]\subseteq N_{d(x)\restriction\beta}$ holds.  
 \end{proof}
 
 Finally, assume that the set $[T]$ is a retract of ${}^\kappa\kappa$. Then Theorem \ref{theorem:ResultsKurepaNoIsolated} shows that $\kappa$ is inaccessible and $S$ is not stationary slim. Define $$U ~ = ~ \Set{u\in{}^{{<}\kappa}2}{(\exists\alpha\in\dom{u} ~ u(\alpha)=1) \longrightarrow u(0)=1}.$$ Then $U$ is a $\kappa$-Kurepa subtree of ${}^{{<}\kappa}\kappa$, the set $[U]$ is a retract of ${}^\kappa\kappa$ and the unique element $x$ of ${}^\kappa\kappa$ with $x(0)=0$ is an isolated point of $[U]$. Let $W$ be the $\kappa$-Kurepa subtree of ${}^{{<}\kappa}\kappa$ obtained from $U$ through the above construction. Then there is $0<\beta_0<\kappa$ such that the set [W] is equal to $$\Set{x\in{}^\kappa 2}{x(0)=1}\cup\Set{x\in{}^\kappa\kappa}{\forall\alpha<\beta_0 ~ x(\alpha)=0 ~ \wedge ~  \betrag{\Set{\alpha<\kappa}{x(\alpha)\neq 0}}<\omega}.$$
 Assume, towards a contradiction, that $\map{r}{{}^\kappa\kappa}{[W]}$ is a map witnessing that $[W]$ is a retract of ${}^\kappa\kappa$. 
 We now inductively construct a sequence $\seq{x_n}{n<\omega}$ of elements of $[W]$ and a strictly increasing sequence $\seq{\beta_n}{n<\omega}$ of ordinals below $\kappa$  such that $x_n(\beta_n)\neq 0$ and $x_n\restriction\beta_{n+1}=x_{n+1}\restriction\beta_{n+1}$ for all $n<\omega$. 
 Let $x_0$ be an arbitrary element of $[W]$ with $x_0(0)=0$ and $x(\beta_0)\neq 0$. 
  If $x_n$ and $\beta_n$ are already constructed, then we can apply Proposition \ref{proposition:SplitPlusInclude} to find $\beta_{n+1}\in(\beta_n,\kappa)$  with $r[N_{x_n\restriction\beta_{n+1}}]\subseteq N_{x_n\restriction\beta_{n+1}}$ and we pick $x_{n+1}\in[W]$  with $x_n\restriction\beta_{n+1}\subseteq x_{n+1}$ and $x_{n+1}(\beta_{n+1})\neq 0$. Pick $x\in{}^\kappa\kappa$ with $x_n\restriction\beta_{n+1}\subseteq x$ for all $n<\omega$. Then $x_n\restriction\beta_{n+1}\subseteq r(x)$ for all $n<\omega$ and hence $r(x)(\beta_n)=x_n(\beta_n)\neq 0$ holds for all $n<\omega$. Since $r(x)(\alpha)=0$ holds for all $\alpha<\beta_0$, this shows that $r(x)\notin[W]$, a contradiction.  
\end{proof}


\section{Open questions}

We end this paper with a compilation of questions left open by the above results. 
 First, note that Theorem \ref{theorem:Main-Negative} shows that by forcing with $\Add{\omega}{\omega_2}$ over $\LL$, we produce a model of set theory in which there exist $\aleph_2$-Kurepa trees and for every such tree $T$, the set $[T]$ is not a continuous image of ${}^{\omega_2}\omega_2$. 
 In addition,  Theorems \ref{theorem:ConsStrengthContImage} and \ref{theorem:KurepaImageInL} show that, in $\LL$, there exist $\aleph_2$-Kurepa trees $T_0$ and $T_1$ such that the set $[T_0]$ is a continuous image of ${}^{\omega_2}\omega_2$ and the set $[T_1]$ is not a continuous image of ${}^{\omega_2}\omega_2$. 
 Therefore, there is only one constellation whose consistency is not settled by our results:

\begin{question} 
Is it consistent with the axioms of $\ZFC$ that there are $\aleph_2$-Kurepa trees and for every such tree $T$, the set $[T]$ is a continuous image of ${}^{\omega_2}\omega_2$? 
\end{question}

Next, note that the only way to apply the above results to obtain models of $\ZFC$ that contain $\aleph_2$-Kurepa trees and have the property that no such tree is a continuous image of ${}^{\omega_2}\omega_2$ is to consider models in which $\CH$ fails. 
 Thus, it is natural to ask the following question:

\begin{question} 
 Does $\CH$ together with the existence of an $\aleph_2$-Kurepa tree imply that there exists an $\aleph_2$-Kurepa subtree $T$ of ${}^{{<}\omega_2}\omega_2$ with the property that the set $[T]$ is: 
 \begin{enumerate}
  \item 
  a continuous image of ${}^{\omega_2}\omega_2$? 

  \item 
  a  retract of ${}^{\omega_2}\omega_2$? 
 \end{enumerate}
\end{question}

Finally, the Kurepa trees constructed in the proofs of the results presented in Section \ref{subsection:PositiveResults} all arise from modifications of canonical Kurepa trees whose existence is ensured by the assumptions of these results. 
 Therefore, it is also interesting to study the descriptive properties of these canonical Kurepa trees themselves.

\begin{question}
 Let $\kappa$ denote an uncountable regular cardinal with the property that $\mu^\omega<\kappa$ holds for all $\mu<\kappa$. 
 \begin{enumerate}
  \item If $T$ is the canonical $\kappa$-Kurepa subtree of ${}^{{<}\kappa}\kappa$ constructed from a $\Diamond_\kappa^+$-sequence (see {\cite[Chapter II, Section 7]{MR597342}}), is the set $[T]$ a continuous image of ${}^\kappa\kappa$? 
  
  \item If $T$ is the canonical $\kappa$-Kurepa subtree of ${}^{{<}\kappa}\kappa$ constructed from a $(\kappa,1)$-morass (see {\cite[Section 1.2]{MR701126}}) or a simplified morass (see \cite{MR771773}), is the set $[T]$ a continuous image of ${}^\kappa\kappa$?  
 \end{enumerate}
\end{question}


\bibliographystyle{amsplain}
\bibliography{references}
 

\end{document}